\documentclass[12pt]{amsart}
\usepackage{fullpage}
\usepackage[utf8]{inputenc}
\usepackage[english]{babel}
\usepackage{mathtools}
\usepackage{amssymb}
\usepackage{amsthm}
\usepackage{mathrsfs}
\usepackage{nicefrac}
\usepackage{upgreek}
\usepackage[T1]{fontenc}
\usepackage{graphicx}
\usepackage{hyperref}

\usepackage[all]{xy}
\usepackage{graphics}
\usepackage{color}
\usepackage[T1]{fontenc}
\usepackage{parskip}

\usepackage{enumitem}

\makeatletter
\newtheorem*{rep@theorem}{\rep@title}
\newcommand{\newreptheorem}[2]{%
\newenvironment{rep#1}[1]{%
 \def\rep@title{#2 \ref{##1}}%
 \begin{rep@theorem}}%
 {\end{rep@theorem}}}
\makeatother

\newcommand{\dl}{\langle\langle}
\newcommand{\dr}{\rangle\rangle}

\newcommand{\C}[1]{{\mathcal #1}}

\newtheorem{theorem}{Theorem}[section]
\newtheorem{lemma}[theorem]{Lemma}
\newtheorem{corollary}[theorem]{Corollary}
\newtheorem{proposition}[theorem]{Proposition}
\theoremstyle{definition}\newtheorem{definition}[theorem]{Definition}
\theoremstyle{remark}\newtheorem{remark}[theorem]{Remark}
\theoremstyle{definition}
\makeatletter\@addtoreset{case}{example}\makeatother
\theoremstyle{definition}

\begin{document}

\title{Strong boundedness of ${\rm SL}_2(R)$ for rings of S-algebraic integers with infinitely many units}

\author{Alexander A. Trost}
\address{Fakult\"{a}t f\"{u}r Mathematik, Ruhr Universit\"{a}t Bochum, D-44780 Bochum, Germany}
\email{Alexander.Trost@ruhr-uni-bochum.de}

\begin{abstract}
A group is called strongly bounded, if the speed with which it is generated by finitely many conjugacy classes has a positive, lower bound only dependent on the number of the conjugacy classes in question rather than the actual conjugacy classes. Earlier papers \cite{KLM,General_strong_bound} have shown that this is a property common to split Chevalley groups defined using an irreducible root system of rank at least $2$ and the ring of all S-algebraic integers and that the situation is dependent on the number theory of $R$ for ${\rm Sp}_4$ and $G_2.$ In this paper, we will show that ${\rm SL}_2(R)$ is also strongly bounded for $R$ the ring of all S-algebraic integers in a number field $K$ with $R$ having infinitely many units and will give a complete account of the existence of small conjugacy classes generating ${\rm SL}_2(R)$ in terms of the prime factorization of the rational primes $2$ and $3$ in $R.$
\end{abstract}

\maketitle

\section{Introduction} % use lowercase except for proper names
\label{intro}

It is a common belief that the group ${\rm SL}_2(R)$ for $R$ the ring of S-algebraic integers in a number field $K$ such that $R$ has infinitely many units, behaves very similar to S-arithmetic split Chevalley groups of higher rank and there have been various results in favor of this philosophy: Classically, there are the results by Serre proving the Congruence Subgroup Property for ${\rm SL}_2(R)$ \cite{MR272790}, which show that the Congruence kernel for ${\rm SL}_2(R)$ and ${\rm SL}_n(R)$ for $n\geq 3$ is identical. Further, it was shown by Carter, Keller and Paige and written up by Morris \cite{MR2357719} that bounded generation not only occurs for higher rank, S-arithmetic Chevalley groups as proven by Carter and Keller \cite{MR704220} and Tavgen \cite{MR1044049}, but also in the case of ${\rm SL}_2(R).$ 
In this paper, we will show that a different property, called \textit{strong boundedness}, previously only shown for S-arithmetic, split Chevalley groups of higher rank, also holds in the case of ${\rm SL}_2(R)$ for $R$ the ring of S-algebraic integers in a number field $K$ such that $R$ has infinitely many units. Roughly speaking, strong boundedness is the property of a group $G$ to admit for each natural number $k$ a global bound $\Delta_k(G)$ such that the word norm $\|\cdot\|_T$ induced by $k$ generating conjugacy classes $T$ of $G$ has diameter at most $\Delta_k(G)$ independent of the specific $T$ in question. This property has first been described by Kedra, Libman and Martin:

\begin{theorem}\cite[Theorem~6.1]{KLM}
\label{KLM_thm}
Let $R$ be a principal ideal domain and let ${\rm SL}_n(R)$ be boundedly generated by elementary matrices for $n\geq 3$ with the diameter 
$\|{\rm SL}_n(R)\|_{EL(n)}$ satisfying $\|{\rm SL}_n(R)\|_{EL(n)}\leq C_n$ for some $C_n\in\mathbb{N}.$ Then ${\rm SL}_n(R)$ is normally generated by the single element $E_{1,n}(1)$ and
\begin{enumerate}
\item{for all finite, normally generating subsets $T$ of $G$, it holds $\|{\rm SL}_n(R)\|_T\leq C_n(4n+4)|T|.$
}
\item{if $R$ has infinitely many maximal ideals, then for each $k\in\mathbb{N}$ there is a finite, normally generating subset $T_k$ of $G$ with $|T_k|=k$ and
$\|{\rm SL}_n(R)\|_{T_k}\geq k$.}
\end{enumerate}
\end{theorem}

So, in particular ${\rm SL}_n(R)$ is strongly bounded for rings of S-algebraic integers $R$ with class number $1.$ This is due to the fact that according to Carter and Keller \cite{MR1044049}, the group ${\rm SL}_n(R)$ is boundedly generated by elementary matrices. In our previous paper \cite{General_strong_bound}, we have shown that this type of statement generalizes to all other S-arithmetic, split Chevalley groups $G(\Phi,R)$ defined using an irreducible root system $\Phi$ of rank at least $2$ and the ring of S-algebraic integers $R$. However surprisingly, it turned out that the behavior of the rank two examples ${\rm Sp}_4(R)$ and $G_2(R)$ are markedly different from the other cases. Namely, to describe the behavior of the $\Delta_k$ for ${\rm Sp}_4(R)$ and $G_2(R)$, it is necessary to understand how the rational prime $2$ factors in the ring of integers $R$ and the situation is generally more involved in this case. So it is unsurprising that the investigation for the case of ${\rm SL}_2(R)$ is even more complicated: The strategies used by Kedra, Libman and Martin to prove Theorem~\ref{KLM_thm} and the similar ones used to derive our results in \cite{General_strong_bound} rely ultimately on the presence of an irreducible, rank $2$ root subsystem of $\Phi$ and the nice interaction of different root elements in $G(\Phi,R)$. So clearly this approach will not work anymore for the example of ${\rm SL}_2(R)$. However, using so-called \textit{self-reproducing elements $C(x)$} of $SL_2(R)$ for $x\in R$ in addition to root elements, one can still prove strong boundedness for ${\rm SL}_2(R)$ assuming $R$ has infinitely many units:

\begin{theorem}\label{main_thm}
Let $R=\C O_S$ be the ring of S-algebraic integers in a number field $K$ such that $R$ has infinitely many units. Then there is a constant $C(R)\in\mathbb{N}$ such that $\Delta_k({\rm SL}_2(R))\leq C(R)k$ holds for all $k\in\mathbb{N}.$
\end{theorem}

To prove this we will reformulate a description of normal subgroups of ${\rm SL}_2(R)$ for $R=\C O_S$ a ring of S-algebraic integers with infinitely many units by Costa and Keller \cite{MR1114610} in terms of a first order theory and then invoke a compactness argument together with the afore-mentioned bounded generation result by Carter, Keller and Paige \cite{MR2357719}.

Further, the behavior of $\Delta_k$ for ${\rm Sp}_4(R)$ and $G_2(R)$ depends on the number theory of $R.$ Namely, our theorem \cite[Theorem~6.3]{General_strong_bound} states that the size of the smallest possible normally generating subset is controlled by the number of certain problematic prime factors of the rational prime $2.$ This dependence on the number theory of $R$ still exists for ${\rm SL}_2(R)$, but the situation is overall more complicated. Not only are there other problematic prime divisors besides certain divisors of $2$, but to a certain extent it is even relevant how the prime divisors of $2$ ramify in $R$:

\begin{theorem}\label{lower_bounds_sl2}
Let $R=\C O_S$ be the ring of S-algebraic integers in a number field $K$ such that $R$ has infinitely many units. Set 
\begin{align*}
&r_1(R):=|\{\C P\mid \C P\text{ is an unramified prime divisor of }2R\text{ with }R/\C P=\mathbb{F}_2\}|,\\
&r_2(R):=|\{\C P\mid \C P\text{ is a ramified prime divisor of }2R\text{ with }R/\C P=\mathbb{F}_2\}|\text{ and}\\
&q(R):=|\{\C P\mid \C P\text{ is a prime divisor of }3R\text{ with }R/\C P=\mathbb{F}_3\}|.
\end{align*}
Then define $v(R):=\max\{2r_2(R)+r_1(R),q(R)\}.$ Then
\begin{enumerate}
\item{$\Delta_k({\rm SL}_2(R))\geq 2k$ holds for all $k\geq v(R)$ and}
\item{$\Delta_k({\rm SL}_2(R))=-\infty$ holds for all $k<v(R).$}
\end{enumerate}
\end{theorem}

This is proven by extracting necessary and sufficient conditions on a set $T$ to normally generate ${\rm SL}_2(R)$ from the proof of Theorem~\ref{main_thm} to subsequently use the congruence subgroup property to describe the problem posed by the prime ideals mentioned in Theorem~\ref{lower_bounds_sl2}. Of particular interest in this context is an epimorphism ${\rm SL}_2(R/\C P^2)\to\mathbb{F}_2\oplus\mathbb{F}_2=(R/\C P^2,+)$ we construct in the proof of Proposition~\ref{ramified_abelianization} for a ramified prime divisor $\C P$ of $2$ with $R/\C P=\mathbb{F}_2$ highlighting the difference between ramified and unramified prime divisors of $2$ in $R.$ At the end of the paper, we also apply Theorem~\ref{lower_bounds_sl2} to the more explicit example of rings of quadratic integers to obtain Corollary~\ref{square_root_v_values}.\\

The paper is structured as follows: In the second section, we introduce the used notation. The third section explains how to derive Theorem~\ref{main_thm} from various technical results and in the fourth section we explain how to deduce these technical results from the structure theory of normal subgroups of ${\rm SL}_2.$ In the fifth section, we will prove Theorem~\ref{lower_bounds_sl2} using the congruence subgroup property and a careful analysis of groups of the form ${\rm SL}_2(R/\C P^L)$ for prime ideals $\C P$ and natural numbers $L\in\mathbb{N}.$ 

\section{Basic definitions and notions}
\label{sec_basic_notions}

First, recall the definition of S-algebraic integers:

\begin{definition}
Let $\C O_K$ be the ring of integers in the number field $K$ and let $S$ be a finite set of non-zero prime ideals in $\C O_K.$ Then the ring 
\begin{equation*}
\C O_S:=\left\{a/b\mid a,b\in\C O_K,b\neq 0,\{\text{ prime divisors of }b\C O_K\}\subset S\right\}
\end{equation*} 
is called the \textit{ring of S-algebraic integers in $K$}.
\end{definition}

\begin{remark}
In the rest of the paper, we will for the sake of brevity, usually not mention the choice of the finite set of non-zero prime divisors $S$ in $\C O_K,$ when talking about the ring of S-algebraic integers $\C O_S$ in a number field $K.$
\end{remark}

We also recall the word norms studied in this paper: 

\begin{definition}
Let $G$ be a group.
\begin{enumerate}
\item{We define $A^B:=B^{-1}AB$ for $A,B\in G$.}
\item{For $T\subset G$, we define $\dl T\dr$ as the smallest normal subgroup of $G$ containing $T.$}
\item{A subset $T\subset G$ is called a \textit{normally generating set} of $G$, if $\dl T\dr=G$.}
\item{For $k\in\mathbb{N}$ and $T\subset G$ denote by 
\begin{equation*}
B_T(k):=\bigcup_{1\leq i\leq k}\{x_1\cdots x_i\mid\forall j\leq i: x_j\text{ conjugate to }A\text{ or } A^{-1}\text{ and }A\in T\}\cup\{1\}.
\end{equation*}
Further set $B_T(0):=\{1\}.$ If $T$ only contains the single element $A$, then we write $B_A(k)$ instead of $B_{\{A\}}(k)$.}
\item{Define for a set $T\subset G$ the \textit{conjugation invariant word norm} $\|\cdot\|_T:G\to\mathbb{N}_0\cup\{+\infty\}$ by 
$\|A\|_T:=\min\{k\in\mathbb{N}_0|A\in B_T(k)\}$ for $A\in\dl T\dr$ and by $\|A\|_T:=+\infty$ for $A\notin\dl T\dr.$ The diameter 
$\|G\|_T={\rm diam}(\|\cdot\|_T)$ of $G$ is defined as the minimal $N\in\mathbb{N}$ such that $\|A\|_T\leq N$ for all $A\in G$ or as $+\infty$ if there is no such $N$.}
\item{Define for $k\in\mathbb{N}$ the invariant 
\begin{equation*}
\Delta_k(G):=\sup\{{\rm diam}(\|\cdot\|_T)|\ T\subset G\text{ with }|T|\leq k,\dl T\dr=G\}\in\mathbb{N}_0\cup\{\pm\infty\}
\end{equation*}
with $\Delta_k(G)$ defined as $-\infty$, if there is no normally generating set $T\subset G$ with $|T|\leq k.$ 
}
\item{The group $G$ is called \textit{strongly bounded}, if $\Delta_k(G)<+\infty$ for all $k\in\mathbb{N}$.}
\end{enumerate}
\end{definition}

Let us next recall the definition of ${\rm SL}_2:$

\begin{definition}
Let $R$ be a commutative ring with $1$. Then ${\rm SL}_2(R):=\{A\in R^{2\times 2}\mid a_{11}a_{22}-a_{12}a_{21}=1\}.$
\end{definition}

Obviously for any commutative ring $R$ and any $x\in R,$ the matrices 
\begin{equation*}
E_{12}(x)=
\begin{pmatrix}
1 & x\\
0 & 1
\end{pmatrix}
\text{ and }
E_{21}(x)=
\begin{pmatrix}
1 & 0\\
x & 1
\end{pmatrix}
\end{equation*}
are elements of ${\rm SL}_2(R).$ They are called \textit{elementary matrices} and we denote the set of elementary matrices $\{E_{12}(x),E_{21}(x)\mid x\in R\}$ by ${\rm EL}.$ The subgroup of ${\rm SL}_2(R)$ generated by ${\rm EL}$ is denoted by $E(2,R).$ Furthermore, for an ideal $I\unlhd R,$ we denote by $E(2,R,I)$ the normal subgroup of $E(2,R)$ generated by the $E(2,R)$-conjugates of elements of $\{E_{12}(x)\mid x\in I\}.$ Further, we define the reduction homomorphism $\pi_I:{\rm SL}_2(R)\to{\rm SL}_2(R/I)$ as the group homomorphism induced by the quotient homomorphism $R\to R/I$ and the subgroup ${\rm SL}_2(R,I)$ of ${\rm SL}_2(R)$ as the kernel of the homomorphism $\pi_I.$ Further, for a unit $u\in R^*$ the element 
\begin{equation*}
h(u)=
\begin{pmatrix}
u & 0\\
0 & u^{-1}
\end{pmatrix}
\end{equation*}
is also an element of ${\rm SL}_2(R).$ We also introduce the concept of self-reproducing elements:

\begin{definition}
Let $R$ be a commutative ring with $1$ and let $x\in R$ be given. Then $C(x):=E_{21}(x)\cdot E_{12}(x)$ is called a \textit{self-reproducing element}. Further, for $X\subset R,$ one defines $C(X)$ as the subgroup of $E(2,R)$ normally generated by the set $\{C(x)\mid x\in X\}.$  
\end{definition}

\begin{remark}
The elements $C(x)$ are called self-reproducing, because the subgroup $C(X)$ normally generated by the self-reproducing elements $\{C(x)\mid x\in X\}$ has the property 
\begin{equation*}
C(X)=[E(2,R),C(X)]
\end{equation*}
according to \cite[Theorem~1.1.1(i)]{MR1114610}. 
\end{remark}

Self-reproducing elements have the following useful property:

\begin{lemma}\label{tech1}
Let $R$ be a commutative ring with $1$, let $T\subset{\rm SL}_2(R), x_1,x_2\in R$ and $k_1,k_2\in\mathbb{N}$ be given. Assume further that $C(x_1)\in B_T(k_1)$ and
$C(x_2)\in B_T(k_2).$ Then $C(x_1+x_2)\in B_T(k_1+k_2).$ 
\end{lemma}

\begin{proof}
First, note for $x,y\in R$ that
\begin{align*}
C(x)^{-1}\cdot C(y)&=[E_{21}(x)\cdot E_{12}(x)]^{-1}\cdot[E_{21}(y)\cdot E_{12}(y)]\\
&=E_{12}(-x)\cdot E_{21}(-x)\cdot E_{21}(y)\cdot E_{12}(y)\\
&=E_{12}(-x)\cdot E_{21}(y-x)\cdot E_{12}(y-x)\cdot E_{12}(x)\\
&=C(y-x)^{E_{12}(x)}.
\end{align*}
This equation implies in particular that $C(-x_1)=[C(x_1)^{-1}]^{E_{12}(-x_1)}\in B_T(k_1)$ and so $C(x_1+x_2)=(C(-x_1)^{-1}\cdot C(x_2))^{E_{12}(x_1)}\in B_T(k_1+k_2)$ follows.
\end{proof}

Next, we recall the definition of level ideals:

\begin{definition}
\label{central_elements_def}
Let $R$ be a commutative ring with $1$ and let 
\begin{equation*}
A=
\begin{pmatrix}
a & b\\
c & d
\end{pmatrix}
\in{\rm SL}_2(R)
\end{equation*}
be given. The \textit{level ideal $l(A)$} is defined as the ideal $(a-d,b,c)\unlhd R.$ For a subset $T\subset{\rm SL}_2(R),$ we set $l(T):=\sum_{A\in T}l(A).$
\end{definition} 

\section{The proof of Theorem~\ref{main_thm}}
\label{proof_main}

The proof of the main theorem operates by extracting increasingly more generic elements from a given collection of conjugacy classes $T$ generating ${\rm SL}_2(R)$ and use them to construct a congruence subgroup. First, we introduce the class of rings used in the investigation:

\begin{definition}\label{many_units}
Let $R$ be a commutative ring with $1$ satisfying the following two properties:
\begin{enumerate}
\item{For each $c\in R-(\{0\}\cup R^*)$ the ring $R/cR$ has (Bass) stable range $1$.}
\item{For each $x\in R-\{0\},$ there is a unit $u\in R$ such that $u^4\neq 1$ and $u^2\equiv 1\text{ mod }xR.$}
\end{enumerate}
Then we call $R$ a \textit{ring with many units.}
\end{definition}

\begin{remark}
\hfill
\begin{enumerate}
\item{A ring $R$ having stable range $1$ is defined as follows: For all $a,b\in R$ with $(a,b)=R$, there is an $x\in R$ such that 
$a+bx$ is a unit in $R.$}
\item{A ring satisfying only Definition~\ref{many_units}(i) is called a \textit{ring of stable range at most $3/2.$}}
\item{Observe that $R$ being a ring with many units is a property describable in first order terms.}
\end{enumerate}
\end{remark}

Next, we note the following important intermediate step that we will prove in the next section:

\begin{theorem}\label{technical_thm1}
Let $R$ be a ring with many units. Then there is a natural number $L_1\in\mathbb{N}$ independent of $R$ such that for each $A\in{\rm SL}_2(R)$ and $x\in l(A)$, one has 
\begin{equation*}
\{C(bx^3)\mid b\in R\}\subset B_A(L_1). 
\end{equation*}
\end{theorem}

Before continuing, we note the following:

\begin{proposition}\label{necessary_cond1}
Let $R$ be a commutative ring with $1$ and let $T\subset{\rm SL}_2(R)$ be a normally generating subset of ${\rm SL}_2(R).$ Then $l(T)=\sum_{A\in T}l(A)=R.$
\end{proposition}

We will omit a proof, because it is virtually identical to the proof of \cite[Lemma~3.3]{General_strong_bound} and rather straightforward. 

\begin{remark}\label{Pi_generators}
For later use, we remark that the set of non-zero prime ideals $\C P$ with $\pi_{\C P}(T)$ consisting only of scalar matrices in ${\rm SL}_2(R/\C P)$ is denoted by $\Pi(T)$ and that $\Pi(T)=\emptyset$ is equivalent to $\sum_{A\in T}l(A)=R.$ 
\end{remark}

Next, we note the following:

\begin{theorem}\label{technical_thm2}
Let $R$ be a commutative ring with $1$ and a unit $u\in R$. Then there is a natural number $L_2$ independent of $R$ and $u$ such that
$\{E_{12}(x\cdot (u^4-1))\mid x\in R\}\subset B_{C(1)}(L_2).$
\end{theorem}

Again, we postpone the proof until the next section. From this, we can deduce:

\begin{theorem}\label{technical_thm3}
Let $R$ be a ring with many units and $T\subset{\rm SL}_2(R)$ finite with $\sum_{A\in T}l(A)=R.$ Then there is a unit $u\in R$ with $u^4\neq 1$ independent of $T$ and a natural number $L_3$ independent of $R,u$ and $T$ such that $\{E_{12}(x(u^4-1))\mid x\in R\}\subset B_T(L_3\cdot |T|).$
\end{theorem}

\begin{proof}
Observe that as $T$ has the property $\sum_{A\in T}l(A)=R$, there are elements $x_A\in l(A)$ for $A\in T$ with $(x_A\mid A\in T)=R.$ So in particular, we have 
$(x_A^3\mid A\in T)=R$. Thus for each $A\in T$ there are $b_A\in R$ with $1=\sum_{A\in T} b_A\cdot x_A^3.$ Observe that Theorem~\ref{technical_thm1} then implies that
$C(b_A\cdot x_A^3)\in B_A(L_1)\subset B_T(L_1)$ for each $A\in T.$ Hence using Lemma~\ref{tech1} repeatedly, we obtain 
\begin{equation*}
C(1)=C(\sum_{A\in T} b_A\cdot x_A^3)\in B_T(L_1\cdot |T|).
\end{equation*} 
But $R$ has many units and so there is a unit $u\in R$ with $u^4\neq 1.$ But then Theorem~\ref{technical_thm2} implies that 
\begin{equation*}
\{E_{12}(x\cdot (u^4-1))\mid x\in R\}\subset B_{C(1)}(L_2)\subset B_T(L_1\cdot L_2\cdot |T|)
\end{equation*}
and so setting $L_3:=L_1\cdot L_2$, we are done.
\end{proof}

Next, we note the following:

\begin{lemma}\label{integers_many_units}
Let $R=\C O_S$ be the ring of S-algebraic integers in a number field $K$ such that $R$ has infinitely many units. Then $R$ is a ring with many units.
\end{lemma}

\begin{proof}
Rings with many units are defined by two properties so showing both properties is enough to prove the lemma. First, to show that $R$ has stable range at most $3/2$ let 
$c\in R-(\{0\}\cup R^*)$ be given. Then the ring $R/cR$ is finite, so in particular it is semi-local and hence has stable range $1$ according to \cite[Lemma~6.4, Corollary~6.5]{MR0174604}. For the second property, let $x\in R-(\{0\}\cup R^*)$ be given. According to Dirichlet's Unit Theorem \cite[Corollary~11.7]{MR1697859}, there is a unit $v\in R$ of infinite order. But observe that $R/xR$ is a finite ring, so there must be a non-trivial power $v^k$ of $v$ such that $v^k+xR=1+xR.$ Setting 
$u:=v^k$, we obtain then that $u^4\neq 1$ and $u^2\equiv 1\text{ mod }xR.$ 
\end{proof}

To finish the proof of Theorem~\ref{main_thm}, we also need

\begin{lemma}
\label{congruence_fin}
Let $R=\C O_S$ be the ring of S-algebraic integers in a number field $K$ such that $R$ has infinitely many units. Further, let $I\unlhd R$ be a non-trivial ideal in $R.$ Also define $Q_I:=\{E_{12}(x)\mid x\in I\}$ and $N_I:=\dl Q_I\dr$ and let $\|\cdot\|_{Q_I}:N_I\to\mathbb{N}\cup\{+\infty\}$ be the conjugation invariant word norm on $N_I$ defined by $Q_I.$  
\begin{enumerate}
\item{
Then the group ${\rm SL}_2(R)/N_I$ is finite.}
\item{Then there is a $K(I,R)\in\mathbb{N}$ such that $\|N_I\|_{Q_I}\leq K(I,R)$.}
\end{enumerate}
\end{lemma}

\begin{proof}
First, note that as $R$ is a ring of S-algebraic integers with infinitely many units, the group ${\rm SL}_2(R)$ is boundedly generated by elementary matrices according to \cite[Theorem~1.2]{MR2357719}. But then the lemma follows in essentially the same way as \cite[Lemma~3.4]{General_strong_bound} by replacing the ideal $2R$ by the ideal $I$.
\end{proof}

We have accumulated the tools to prove Theorem~\ref{main_thm} now:

\begin{proof}
According to Lemma~\ref{integers_many_units}, the ring $R$ of S-algebraic integers is a ring with many units. Let $T$ be a finite, normally generating subset of 
${\rm SL}_2(R).$ Then according to Lemma~\ref{necessary_cond1}, we have $\sum_{A\in T}l(A)=R$. So Theorem~\ref{technical_thm3} implies that there is a unit 
$u\in R$ independent of $T$ with $u^4\neq 1$ and 
\begin{equation*}
\{E_{12}(x(u^4-1))\mid x\in R\}\subset B_T(L_3\cdot |T|)
\end{equation*}
with $L_3$ independent of $R,u$ and $T.$ Define $I:=(u^4-1)\unlhd R$ and consider the group $G:={\rm SL}_2(R)/N_I$ together with the quotient map 
$\pi:{\rm SL}_2(R)\to G.$ Then consider the set 
\begin{equation*}
E(G):=\{\bar{S}\subset G\mid \bar{S}\text{ normally generates }G\}.
\end{equation*}
But according to Lemma~\ref{congruence_fin}(i), the group $G$ is finite and hence so is $E(G).$ Thus there is a constant $M:=M(I)\in\mathbb{N}$ such that for each 
$S\subset{\rm SL}_2(R)$ with $\pi(S)\in E(G)$ and each $A\in{\rm SL}_2(R)$, there are $s_1,\dots,s_{M(I)}\in S\cup S^{-1}\cup\{I_2\}$ and 
$X_1,\dots,X_{M(I)}\in{\rm SL}_2(R)$ with 
\begin{equation*}
\pi(A)=\pi\left(\prod_{j=1}^{M(I)} s_j^{X_j}\right).
\end{equation*}  
Hence $A(\prod_{j=1}^{M(I)} s_j^{X_j})^{-1}\in N_I$ and so setting $\|N_I\|_T:={\rm sup}\{\|g\|_T|g\in N\}\in\mathbb{N}\cup\{+\infty\}$ implies 
\begin{equation*}
\|A\|_T\leq\|\prod_{j=1}^{M(I)} s_j^{X_j}\|_T+\|N_I\|_T\leq M(I)\cdot\max\{\|s\|_T|\ s\in S\}+\|N_I\|_T.
\end{equation*}
This implies $\|{\rm SL}_2(R)\|_T\leq M(I)\max\{\|s\|_T|\ s\in S\}+\|N_I\|_T$ for all $S\subset{\rm SL}_2(R)$ with $\pi(S)\in E(G).$
Next, note that as $T$ normally generates ${\rm SL}_2(R)$, we obtain $\pi(T)\in E(G)$ and for $t\in T$, one clearly has $\|t\|_T\leq 1$ and thus $\|{\rm SL}_2(R)\|_T\leq M(I)+\|N_I\|_T.$ But $M(I)$ does not not depend on $T$ and hence it suffices to bound $\|N_I\|_T$ linearly in $|T|$ to finish the proof of Theorem~\ref{main_thm}. But note that according to Theorem~\ref{technical_thm3}, there is an $L_3\in\mathbb{N}$ such that for all $q\in Q_I$, one has $\|q\|_T\leq L_3|T|.$ Further by Lemma~\ref{congruence_fin}(ii), there is a $K(I,R)\in\mathbb{N}$ such that $\|N_I\|_{Q_I}\leq K(I,R)$. This implies $\|N_I\|_T\leq K(I,R)L_3|T|$ and this finishes the proof.
\end{proof}

Next, we want to derive a criterion for a subset $T$ to normally generate ${\rm SL}_2(R)$ from the previous discussion. First, let $T$ be a subset of ${\rm SL}_2(R)$ with $\Pi(T)=\emptyset$ and let $N$ be the normal subgroup generated by $T.$ Note that we derived in the proof of Theorem~\ref{technical_thm3} that $C(1)$ is an element of $N.$ However, the argument can equally well be used to derive that $\{C(x)\mid x\in R\}$ is a subset of $N.$ However, note that as $C(x)=E_{21}(x)E_{12}(x)$, this implies that $E_{21}(-x)N=E_{12}(x)N$ holds in ${\rm SL}_2(R)/N$ for all $x\in R.$ But ${\rm SL}_2(R)$ is generated by the set ${\rm EL}$ of elementary matrices. Hence ${\rm SL}_2(R)/N$ is in fact an abelian group. Thus $T$ normally generates ${\rm SL}_2(R)$ if it has the additional property that $T$ maps to a generating set of the abelianization $H_1({\rm SL_2}(R)).$ But clearly if on the other hand $T$ is a normally generating subset of ${\rm SL}_2(R),$ then $\Pi(T)=\emptyset$ and $T$ mapping onto a generating subset of $H_1({\rm SL}_2(R))$ must hold. Thus we have proven:

\begin{corollary}\label{normal_gen_criterion}
Let $R=\C O_S$ be the ring of S-algebraic integers of a number field $K$ such that $R$ has infinitely many units and let $T\subset{\rm SL}_2(R)$ be given. Then $T$ normally generates ${\rm SL}_2(R)$ if and only if $\Pi(T)=\emptyset$ and $T$ maps onto a generating set of the abelianization of ${\rm SL}_2(R).$
\end{corollary}

\section{Compactness arguments, radices and the proofs of Theorem~\ref{technical_thm1} and Theorem~\ref{technical_thm2}}

The strategy to prove Theorem~\ref{technical_thm1} is to interpret it as a consequence of the inconsistency of a certain first order theory. In order to do this we need some more terminology. 

\begin{remark}
Much of the terminology and the main statements in this subsection are due to the paper \cite{MR1114610} by Costa and Keller and we highly encourage to read it.
\end{remark}

First, we introduce the concept of radices defined in \cite{MR1114610}:

\begin{definition}
Let $R$ be a commutative ring with $1$ and $P\subset R$ a subgroup of $(R,+).$ Then $P$ is called a \textit{radix}, if it satisfies the following two conditions:
\begin{enumerate}
\item{$(a^3-a)x^2+a^2x$ is an element of $P$ for all $x\in P$ and $a\in R.$}
\item{$ax^3$ is an element of $P$ for all $x\in P$ and $a\in R.$}
\end{enumerate}
\end{definition}

Next, we define a special map called $\rho:$

\begin{definition}\label{rho_def}
Let $R$ be a commutative ring with $1$. Then the map $\rho:{\rm SL}_2(R)\to R$ is defined by
\begin{equation*}
\rho
\begin{pmatrix}
a & b\\
c & d
\end{pmatrix}
=a^2-1+ab.
\end{equation*}
Similarly, the maps $\rho_T,\rho_{-1}$ and $\rho_{-T}:{\rm SL}_2(R)\to R$ are defined as follows for $A\in{\rm SL}_2(R)$:
\begin{enumerate}
\item{$\rho_T(A):=\rho(A^T),$}
\item{$\rho_{-1}(A):=\rho(A^{-1})$ and}
\item{$\rho_{-T}(A):=\rho(A^{-T}).$}
\end{enumerate}
\end{definition}

We also define the following ideal ${\rm vn}_2(R),$ the so-called \textit{booleanizing ideal of $R$}:

\begin{definition}
Let $R$ be a commutative ring with $1.$ Then ${\rm vn}_2(R)$ is defined as the ideal
\begin{equation*}
(x^2-x\mid x\in R).
\end{equation*}
\end{definition}

Using the map $\rho$ and the ideal ${\rm vn}_2(R)$, one can define the following subgroups:

\begin{definition}
Let $R$ be a commutative ring with $1$ and $P\subset R$ a radix. Then the subgroup $G(P)$ of ${\rm SL}_2(R)$ is defined as:
\begin{align*}
G(P):=
\{
A=
\begin{pmatrix}
a & b\\
c & d
\end{pmatrix}\in{\rm SL}_2(R)
\mid &\rho(A),\rho_T(A),\rho_{-1}(A),\rho_{-T}(A)\in P\text{ and}\\
&(a^2-1)\cdot{\rm vn}_2(R)+(d^2-1)\cdot{\rm vn}_2(R)\subset P
\}.
\end{align*}
Further, for $J\unlhd R$ a non-trivial ideal and a subgroup $U$ of $(R/J)^*$, slightly abusing notation, the subgroup $G(J,U)$ is defined as
\begin{equation*}
G(J,U):=\{A\in{\rm SL}_2(R)\mid A\equiv uI_2\text{ mod }J\text{ for some }u\in U\}
\end{equation*}
Last, define $G(J,U,P):=G(P)\cap G(J,U).$
\end{definition}

Next, we define another subgroup associated to a normal subgroup $N$ of ${\rm SL}_2(R):$

\begin{definition}
Let $R$ be a commutative ring with $1$ and let $N$ be a normal subgroup of ${\rm SL}_2(R).$ Then define $\rho(N)$ as the subgroup of $(R,+)$ generated by the set
$\{\rho(A)\mid A\in N\}$ and $U(N)$ as the subgroup of $(R/l(N))^*$ generated by all $u\in (R/l(N))^*$ such that (slightly abusing notation) there is an $A\in N$ with 
\begin{equation*}
A\equiv uI_2\text{ mod }l(N).
\end{equation*}
Then $G(N)$ is defined as $G(l(N),U(N),\rho(N)).$
\end{definition}

\begin{remark}
Implicit in this definition is the fact that $\rho(N)$ is a radix \cite[Theorem~2.2.2]{MR1114610}.
\end{remark}

Finally, we can state the following version of \cite[Theorem~3.1.3]{MR1114610}:

\begin{theorem}\label{sandwich_thm}
Let $R$ be a ring with many units and $N$ a normal subgroup of ${\rm SL}_2(R)$ with $l(N)\neq (0).$ Then the following chain of inclusions holds.
\begin{equation*}
[E(2,A),G(N)]\subset N\subset G(N).
\end{equation*}
\end{theorem}

\begin{proof}
We will only sketch the proof, because it is largely identical to the one of \cite[Theorem~3.1.3]{MR1114610} with the main difference being technical details. Costa and Keller prove their theorem by showing that $N$ and $R$ satisfy the assumptions of a different, technical theorem \cite[Theorem~3.1.2]{MR1114610} which states that a chain of inclusions as required exists for any commutative ring $R$ and any normal subgroup $N$ of ${\rm SL}_2(R)$ under the assumption that there is an ideal $J_0\unlhd R$ not equal to $R$ satisfying two properties:
\begin{enumerate}
\item{$[E(2,R),{\rm SL}_2(R,J_0)]\subset C(J_0)\subset N$ and}
\item{$R/J_0$ has stable range $1.$ %$SR_2(R/J_0,l(N)/J).$
}
\end{enumerate} 
However, note that as $l(N)\neq 0,$ there is an element 
\begin{equation*}
A=
\begin{pmatrix}
a & b\\
c & d
\end{pmatrix} 
\end{equation*}
of $N$ such that $b\neq 0.$ But $R$ has many units by assumption and hence there is a unit $u\in R$ such that $u^4\neq 1$ and $u^2\equiv 1\text{ mod }bR.$ But then 
\cite[Lemma~6.5]{MR2357719} implies that $E(2,R,(u^4-1)R)$ is a subgroup of $N.$ Setting $J_0:=(u^4-1)R,$ it is clear that $J_0$ is a non-zero ideal in $R$ and we may assume wlog that $J_0\neq R.$ However, the ring $R$ has many units and so the ring $R/J_0$ has stable range $1$.% and so $SR_2(A/J_0,l(N)/J_0)$ must hold. 
This settles the second condition. Next, note that $C(J_0)\subset E(2,R,J_0)$ holds. But note that $C(J_0)$ itself also has the property that $l(C(J_0))\neq (0),$ because $C(u^4-1)$ is not a scalar element. Thus using the property of $R$ to have many units together with \cite[Lemma~6.5]{MR2357719} again, we can find another non-trivial ideal $J$ such that $E(2,R,J)\subset C(J_0).$ But then \cite[p.240]{MR0249491} immediately implies that 
\begin{equation*}
{\rm SL}_2(R,J_0)=E(2,R,J_0)\cdot{\rm SL}_2(R,J).  
\end{equation*}
But this then implies 
\begin{equation*}
[E(2,R),{\rm SL}_2(R,J_0)]=[E(2,R),E(2,R,J_0)\cdot{\rm SL}_2(R,J)]=[E(2,R),E(2,R,J_0)]\cdot[E(2,R),{\rm SL}_2(R,J)].
\end{equation*}
However, note that $[E(2,A),E(2,R,J_0)]=C(J_0)$ holds according to \cite[Theorem~1.1.1]{MR1114610}. Hence it suffices to show that $[E(2,R),{\rm SL}_2(R,J)]\subset C(J_0)$ holds to finish the proof of the first condition above and so the proof of the theorem. But already
\begin{equation*}
[E(2,R),{\rm SL}_2(R,J)]\subset E(2,R,J)
\end{equation*}
holds according to \cite[Corollary~5.23]{MR2357719} and this finishes the proof together with the aforementioned inclusion $E(2,R,J)\subset C(J_0).$
\end{proof}

This implies the following:

\begin{corollary}\label{sandwich_cor}
Let $R$ be a ring with many units and $N$ a normal subgroup of ${\rm SL}_2(R)$. Then for each $x\in l(N)$ and $b\in R$, the element  
$C(x^3b)$ is an element of $N.$
\end{corollary}

\begin{proof}
Let $x\in l(N)$ and $b\in R$ be given. Then we first show that $E_{12}(x^3b)$ is an element of $G(N).$ First, it is clear that $E_{12}(x^3b)$ is an element of 
$G(l(N),U(N)),$ because $x$ is an element of $l(N).$ Second, we observe that 
\begin{align*}
&\rho(E_{12}(x^3b))=1^2-1+1\cdot x^3b=x^3b,\rho_{-1}(E_{12}(x^3b))=-x^3b\\
&\rho_T(E_{12}(x^3b))=\rho(E_{21}(x^3b))=1^2-1+1\cdot 0=0\text{ and }\rho_{-T}(E_{12}(x^3b))=0.
\end{align*}
and that $(1^2-1)\cdot{\rm vn}_2(R)+(1^2-1)\cdot{\rm vn}_2(R)=(0).$ Thus $E_{12}(x^3b)$ being an element of $G(\rho(N))$ is equivalent to $x^3b$ being an element of 
$\rho(N).$ But this follows from \cite[Theorem~2.2.2]{MR1114610}. Furthermore, \cite[Theorem~1.1.1]{MR1114610} implies that $C(x^3b)$ is an element of $[E(2,R),\dl E_{12}(x^3b)\dr]$ and finishes the proof.
\end{proof}

Using Corollary~\ref{sandwich_cor}, we can finally prove Theorem~\ref{technical_thm1}:

\begin{proof}
Let a language $\C L$ with the relation symbols, constants and function symbols 
\begin{equation*}
(\C R,0,1,+,\times,(a_{i,j})_{1\leq i,j\leq 2},t,u,v,w,s) 
\end{equation*}
and a further function symbol $\cdot^{-1}:\C R^{2\times 2}\to\C R^{2\times 2}$ be given. Note that we use capital letters to denote matrices of variables (or constants) in the language in the following. For example the symbol $\C A$ denotes the $2\times 2$-matrix of constants $(a_{i,j})$ and $X$ commonly refers to matrices of $2\times 2$ variables in $\C L$. We also use the notation $X^{-1}:=^{-1}(X)$. However this is only a way to simplify notation, because first order sentences about matrices can always be reduced to first order sentences about their entries. Furthermore, we use the function symbol $P:\C R^{2\times 2}\to\C R$ defined in the language $\C L$ by
\begin{equation*}
P(X)=x_{11}\times x_{44}-x_{12}\times x_{21}-1.
\end{equation*}

Next, we define the following first order theory $\C T$:
\begin{enumerate}
\item{Sentences forcing the universe $R:=\C R^{\C M}$ of each model $\C M$ of $\C T_{kl}$ to be a commutative ring with respect to the functions 
$+^{\C M},\times^{\C M}$ and with $0^{\C M},1^{\C M}$ being $0$ and $1$.}
\item{The sentence $\forall a,b,c:(\exists x,y,z: a\times x+b\times y+c\times z=1)\rightarrow(\exists x,y,z:(a+b\times x)\times y+c\times z=1).$}
\item{The equation of constants $t=u\times(a_{11}-a_{44})+v\times a_{12}+w\times a_{21}.$}
\item{The sentence $\forall x\neq 0:\exists c\exists y\exists z:(c^4\neq 0)\wedge (c^2-1=y\times x)\wedge (c\times z=1).$}
\item{The sentence $P(\C A)=0$.}
\item{The sentence $\forall X:(P(X)=0)\rightarrow (XX^{-1}=I_2),$ where $I_2$ denotes the unit matrix in $\C R^{2\times 2}$ with entries the constant symbols $0,1$ as appropriate.}
\item{A family of sentences $(\theta_r)_{r\in\mathbb{N}}$ as follows: 
\begin{align*}
\theta_r:&\bigwedge_{1\leq k\leq r}\forall X_1,\dots,X_r,\forall e_1,\dots,e_r\in\{0,1,-1\}:\\
&((P(X_1)=\cdots=P(X_r)=0)\rightarrow 
(C(t^3\times s))\neq (\C A^{e_1})^{X_1}\cdots(\C A^{e_r})^{X_r})
\end{align*}
Here $\C A^{1}:=\C A,\C A^{-1}:=\C A^{-1}$ and $\C A^{0}:=I_2.$
}
\end{enumerate} 
Next, we will show that this theory $\C T$ is inconsistent. In order to see this let $\C M$ be a model of the sentences in (1) through (6). We denote the ring 
$\C R^{\C M}$ by $R$ and the element $\C A^{\C M}$ by $A.$ Note that the sentences (2) and (4) imply that $R$ is a ring with many units. Note that sentence (5) forces $A$ to be an element of ${\rm SL}_2(R)$ and that sentence (3) forces $t^{\C M}$ to be an element of $l(A).$ Abusing notation, we will denote $t^{\C M}$ by $t.$ Next, consider the normal subgroup $N$ of ${\rm SL}_2(R)$ normally generated by $A.$ But then Corollary~\ref{sandwich_cor} implies that $C(t^3\times s^{\C M})$ is an element of $N$ and hence there are $X_1,\dots,X_r\in{\rm SL}_2(R)$ and $e_1\dots,e_r\in\{0,1,-1\}$ such that 
\begin{equation*}
C(t^3\times s^{\C M})=(A^{e_1})^{X_1}\cdots(A^{e_r})^{X_r}. 
\end{equation*}
But this contradicts the statement $\theta_r^{\C M}.$ Hence the theory $\C T$ is inconsistent and thus according to G\"{o}del's Compactness Theorem \cite[Satz~3.2]{MR2596772} already a finite sub-theory $\C T_0$ is inconsistent. Next, let $L_1$ be the largest natural number $r\in\mathbb{N}$ such that $\theta_r\in\C T_0.$ But note that $\{(1)-(6),\theta_{r+1}\}\vdash\theta_r$ holds for all natural numbers $r\in\mathbb{N}$ and so already the theory $\C T_1$ consisting of all the sentences in (1) through (6) and the single sentence $\theta_{L_1}$ is inconsistent.

However, if $R$ is any ring with many units, $A$ an element of ${\rm SL}_2(R)$, $t$ an element of $l(A)$ and $b$ an arbitrary element of $R$, then this setup gives a model $\C M$ for all the sentences in (1) through (6) and hence this setup must violate the statement $\theta_{L_1}^{\C M}.$ Thus we obtain $C(x^3b)$ is an element of $B_A(L_1).$ This finishes the proof.
\end{proof}

We finish this section with the proof of Theorem~\ref{technical_thm2}. Again, this proof is done with a model-theoretic compactness argument similar to the proof of Theorem~\ref{technical_thm1}. For this reason, we will only give a brief sketch. 

\begin{proof}
Again, one defines a language $\C L$ with the relation symbols, constants and function symbols $(\C R,0,1,+,\times,t,u)$ and then defines a theory $\C T$, which enforces that for each model $\C M,$ the universe $\C R^{\C M}$ is a commutative ring. Additionally the theory $\C T$ should contain sentences that ensure that for each $k\in\mathbb{N},$ the element $E_{12}(t(u^4-1))$ can not be written as a product of $k$ conjugates of $C(1)$ or $C(1)^{-1}$ in ${\rm SL}_2(R).$ However, one then easily obtains that the theory $\C T$ is inconsistent, because $E(R,2,(u^4-1)R)$ is contained in the subgroup of ${\rm SL}_2(R)$ normally generated by $C(1)$ according to \cite[Lemma~6.5]{MR2357719}. Then using compactness this implies in particular that there is an $L_2\in\mathbb{N}$ such that $E_{12}(t(u^4-1))$ can be written as a product of at most $L_2$ conjugates of $C(1)$ or $C(1)^{-1}$ for any $t\in R.$ 
\end{proof}

\section{Normally generating subsets of ${\rm SL}_2(R)$ and the proof of Theorem~\ref{lower_bounds_sl2}}

To prove Theorem~\ref{lower_bounds_sl2}, we will use Corollary~\ref{normal_gen_criterion}. Note that it requires the study of the abelianization of ${\rm SL}_2(R).$ Hence in the first subsection, we will determine the abelianization of ${\rm SL}_2(R)$ and use it to prove the second part of Theorem\ref{lower_bounds_sl2}. Then in the second subsection, we will use the description of the abelianization to prove the first part of Theorem~\ref{lower_bounds_sl2}.

\subsection{The abelianization of ${\rm SL}_2(R)$}

Note that the commutator subgroup of ${\rm SL}_2(R)$ is a non-central subgroup of ${\rm SL}_2(R).$ Assuming wlog. that ${\rm SL}_2(R)$ is not perfect, Serre's solution of the congruence subgroup problem for ${\rm SL}_2(R)$ \cite{MR272790}, implies the existence of a maximal, non-trivial finite index ideal $I$ of $R$ such that the abelianization map of ${\rm SL}_2(R)$ factors through ${\rm SL}_2(R/I).$ Let $I$ factor into prime ideals as $I=\prod_{j=1}^n\C X_j^{y_j}$ for prime ideals $\C X_j$ and positive integers $y_j.$ Then using the Chinese Remainder Theorem ${\rm SL}_2(R/I)$ factors as $\prod_{j=1}^n{\rm SL}_2(R/\C X_j^{y_j}).$ Hence to determine the abelianization of ${\rm SL}_2(R)$ it suffices to determine the abelianization of the various ${\rm SL}_2(R/\C X_j^{y_j}).$ Most of the prime ideals however give rise to perfect ${\rm SL}_2$ according to the following lemma: 

\begin{lemma}
Let $R$ be a ring of 
 integers and $\C X$ a non-zero prime ideal of $R$ with $|R/\C X|\geq 4$ and $L$ a natural number. Than ${\rm SL}_2(R/\C P^L)$ is a perfect group.
\end{lemma}

\begin{proof}
First, note that there has to be a unit $\bar{u}\in R/\C X^L=:\bar{R}$ such that $\bar{u}^2-1$ is also a unit. This is the case, because an element of $R/\C X^L$, which is not a unit must be the image of an element of $\C X\subset R.$ Hence if there is no unit $\bar{u}\in\bar{R}$ as required, then $y(y^2-1)=y(y-1)(y+1)$ would have to be an element of $\C X$ for each element $y$ of $R.$ Phrased differently, the field $R/\C X$ would have at most three elements, which contradicts $|R/\C X|\geq 4.$ Then for any $\bar{x}\in\bar{R},$ consider the element 
$(h(\bar{u}),E_{12}(\bar{x}\cdot(\bar{u}^2-1)^{-1}))$ of the commutator subgroup of ${\rm SL}_2(\bar{R}).$ We obtain 
\begin{align*}
(h(\bar{u}),E_{12}(\bar{x}\cdot(\bar{u}^2-1)^{-1})))&=h(\bar{u})\cdot E_{12}(\bar{x}\cdot(\bar{u}^2-1)^{-1}))\cdot h(\bar{u})^{-1}\cdot E_{12}(-\bar{x}\cdot(\bar{u}^2-1)^{-1})\\
&=E_{12}(\bar{u}^2\bar{x}\cdot(\bar{u}^2-1)^{-1})\cdot E_{12}(-\bar{x}\cdot(\bar{u}^2-1)^{-1}))=E_{12}(\bar{x}).
\end{align*}
So each elementary matrix of ${\rm SL}_2(\bar{R})$ is contained in the commutator subgroup of ${\rm SL}_2(\bar{R}).$ But by \cite[p.240]{MR0249491}, the set of elementary matrices generates ${\rm SL}_2(\bar{R})$ and we are done.
\end{proof}

So we can assume that the aforementioned ideal $I$ only has prime divisors $\C X$ such that $R/\C X$ has at most three elements. Next, set $r_1:=r_1(R),r_2:=r_2(R)$ and $q:=q(R)$ and 
\begin{enumerate}
\item{let $\C P_1,\dots,\C P_{r_2}$ be the ramified prime divisors of $2R$ with $R/\C P_j=\mathbb{F}_2$,}
\item{let $\C Q_1,\dots,\C Q_{r_1}$ be the unramified prime divisors of $2R$ with $R/\C Q_i=\mathbb{F}_2$ and} 
\item{let $\C K_1,\dots,\C K_q$ be the prime divisors of $3R$ with $R/\C K_l=\mathbb{F}_3$}.
\end{enumerate}

By the previous discussion, there are natural numbers $p_1,\dots,p_{r_2},q_1,\dots,q_{r_2}$ and $k_1,\dots,k_q$ such that $I=(\prod_{j=1}^{r_2}\C P_j^{p_j})\cdot(\prod_{i=1}^{r_1}\C Q_i^{q_i})\cdot(\prod_{l=1}^q\C K_l^{k_l}).$ First, we are going to determine the abelianization of ${\rm SL}_2(R/\C Q_i^{q_i}).$ To find this abelianization, we note:

\begin{lemma}\label{lemma_iso_unramified}
Let $R=\C O_S$ be the ring of S-algebraic integers of a number field $K$, $\C Q$ be an unramified prime divisor of $2R$ with $R/\C Q=\mathbb{F}_2$ and $L\in\mathbb{N}.$ Then $\mathbb{Z}_{2^L}\to R/\C Q^L,n+2^L\mathbb{Z}\mapsto n+\C Q^L$ defines an isomorphism.
\end{lemma}

\begin{proof}
First, note that $2\in\C Q$ and hence $2^L\mathbb{Z}$ is a subset of $\C Q^L.$ This implies that the inclusion $\mathbb{Z}\to R$ induces the ring homomorphism 
$q:\mathbb{Z}_{2^L}\to R/\C Q^L,n+2^L\mathbb{Z}\mapsto n+\C Q^L.$ However, note that both $\mathbb{Z}_{2^L}$ and $R/\C Q^L$ have $2^L=|R/\C Q|^L$ many elements. Thus to prove that $q$ is an isomorphism it suffices to prove that the order of $1=1+\C Q^L$ in the group $(R/\C Q^L,+)$ is $2^L$. To this end, let $j$ be the order of $1$ as a group element in $(R/\C Q^L,+).$ Note that $j$ must be a power of $2,$ say $j=2^i.$ This implies $2^i+\C Q^L=0$ and so $2^i$ is an element of $\C Q^L.$ However, $\C Q$ is an unramified prime divisor of $2R$ and hence $\C Q$ can divide $2^i$ at most $i$ times. Thus $j=2^L$ must hold and this finishes the proof.    
\end{proof}

But using this lemma one can identify the abelianization of ${\rm SL}_2(R/\C Q_i^{q_i}):$

\begin{proposition}\label{abelianization_umramified}
Let $R=\C O_S$ be the ring of S-algebraic integers of a number field $K$, $\C Q$ be an unramified prime divisor of $2R$ with $R/\C Q=\mathbb{F}_2$ and $L\geq 2.$ Then the abelianization map of ${\rm SL}_2(R/\C Q^L)$ is the unique epimorphism ${\rm SL}_2(R/\C Q^L)\to R/\C Q^2$ mapping $E_{12}(x+\C Q^L)$ to $x+\C Q^2$ for all $x\in R.$
\end{proposition}

\begin{proof}
According to Lemma~\ref{lemma_iso_unramified}, the additive group of the ring $R/\C Q^L$ is generated by the unit $1+\C Q^L.$ However, an observation by Jun Morita \cite[Fact~2]{MR1335315} states that the abelianization of ${\rm SL}_2(S)$ is a quotient of $\mathbb{Z}_{12}\oplus\mathbb{Z}_6$ for any ring $S$ whose additive group is generated by two units of $S.$ Thus the abelianization of ${\rm SL}_2(R/\C Q^L)$ is a quotient of $\mathbb{Z}_{12}\oplus\mathbb{Z}_6.$ 
But on the other hand, the group ${\rm SL}_2(R/\C Q^L)$ is generated by the set $\{E_{12}(\bar{x}),E_{21}(\bar{x})\mid\bar{x}\in R/\C Q^L\}$ according to \cite[p.240]{MR0249491} and
\begin{equation*}
\begin{pmatrix}
0 & 1\\
-1 & 0
\end{pmatrix}
\cdot
E_{12}(\bar{x})
\cdot
\begin{pmatrix}
0 & -1\\
1 & 0
\end{pmatrix}
=E_{21}(\bar{x})
\end{equation*}
holds for all $\bar{x}\in R/\C Q^L.$ Hence the abelianization of ${\rm SL}_2(R/\C Q^L)$ is a quotient of $(R/\C Q^L,+)$ and so is in particular a cyclic group with at most four elements according to Lemma~\ref{lemma_iso_unramified}. But together with the abelianization being a quotient of $\mathbb{Z}_{12}\oplus\mathbb{Z}_6$, this implies that the abelianization is the cyclic group with either $1,2$ or $4$ elements. But consider the reduction homomorphism ${\rm SL}_2(R/\C Q^L)\to{\rm SL}_2(R/\C Q^2)={\rm SL}_2(\mathbb{Z}_4)$. We leave it as an exercise to the reader to check that there is an epimorphism ${\rm SL}_2(\mathbb{Z}_4)\to\mathbb{Z}_4$ mapping $E_{12}(n+4\mathbb{Z})$ to $n+4\mathbb{Z}$. Using the isomorphism $\mathbb{Z}_4\to R/\C Q^2,n+4\mathbb{Z}\mapsto n+\C Q^2$ from Lemma~\ref{lemma_iso_unramified}, this finishes the proof.
\end{proof}

Next, we are going to determine the abelianization of ${\rm SL}_2(R/\C P_j^{p_j}).$ We note: 

\begin{lemma}\label{factoring_ramified}
Let $R=\C O_S$ be the ring of S-algebraic integers of a number field $K$, let $\C P$ be a ramified prime divisor of $2R$ with $R/\C P=\mathbb{F}_2$ and let $L\geq 2$ be a natural number. Then the abelianization homomorphism of ${\rm SL}_2(R/\C P^L)$ factors through ${\rm SL}_2(R/\C P^2)$. In particular, the abelianization of ${\rm SL}_2(R/\C P^L)$ is a quotient of $R/\C P^2$ and so has at most $|R/\C P^2|=4$ elements.
\end{lemma}

\begin{proof}
Choose an element $x$ of $\C P-\C P^2$ and let $\bar{x}$ be its image in $\bar{R}:=R/\C P^L.$ Then $1+\bar{x}$ is a unit in $\bar{R}.$ Further, let $\bar{a}\in\bar{R}$ be arbitrary. Then
\begin{equation*}
[h(1+\bar{x}),E_{12}(\bar{a})]=E_{12}((1+\bar{x})^2\bar{a})\cdot E_{12}(-\bar{a})=E_{12}((2\bar{x}+\bar{x}^2)\cdot\bar{a})
\end{equation*}
is an element of the commutator subgroup $[{\rm SL}_2(\bar{R}),{\rm SL}_2(\bar{R})].$ But note that $2x+x^2$ is an element of $\C P^2,$ but not of $\C P^3.$ This is the case, because $\C P$ is a ramified prime divisor of $2R$ and so $2x\in\C P^3$ holds. Thus $2x+x^2\in\C P^3$ would imply $x^2\in\C P^3$ and so $x\in\C P^2,$ a contradiction. But this implies that $2\bar{x}+\bar{x}^2$ generates the principal ideal $\bar{\C P}^2$ for $\bar{\C P}$ the image of $\C P$ in $\bar{R}.$ So in particular, the commutator subgroup of ${\rm SL}_2(\bar{R})$ contains the normal subgroup $N:=\dl E_{12}(\bar{a})\mid\bar{a}\in\bar{\C P}^2\dr.$ But according to \cite[p.240]{MR0249491}, this normal subgroup is the kernel of the epimorphism $\pi_{\bar{\C P^2}}:{\rm SL}_2(\bar{R})\to{\rm SL}_2(\bar{R}/\bar{\C P}^2)={\rm SL}_2(R/\C P^2).$ Thus the abelianization homomorphism of ${\rm SL}_2(\bar{R})$ indeed factors through ${\rm SL}_2(R/\C P^2)$ and we may assume $L=2.$ But as in the proof of Proposition~\ref{abelianization_umramified}, this implies that the abelianization of ${\rm SL}_2(\bar{R})$ is a quotient of $(\bar{R},+)=(R/\C P^2,+)$. But $|R/\C P^2|=|R/\C P|^2=4.$ 
\end{proof}

To determine the abelianization of ${\rm SL}_2(R/\C P^2)$, we however need a technical proposition. To state it, let $R=\C O_S$ be the ring of S-algebraic integers of a number field $K$ and $\C P$ a ramified prime divisor of $2R$ with $R/\C P=\mathbb{F}_2.$ Then set $\bar{R}:=R/\C P^2$ and $\bar{\C P}$ the image of $\C P$ in $\bar{R}.$ Then consider the group ${\rm SL}_2(\bar{R})$ and note that 
\begin{equation*}
K:=\left\{
\begin{pmatrix}
a & b\\
c & d
\end{pmatrix}\in{\rm SL}_2(\bar{R})
\mid a-1,d-1,b,c\in\bar{\C P}
\right\}
\end{equation*}
is the kernel of the reduction homomorphism
\begin{equation*}
\pi_{\bar{\C P}}:{\rm SL}_2(\bar{R})\to{\rm SL}_2(\bar{R}/\bar{\C P})={\rm SL}_2(R/\C P)={\rm SL}_2(\mathbb{F}_2).
\end{equation*}
This yields a short exact sequence 
\begin{equation*}
1\longrightarrow K\longrightarrow{\rm SL}_2(\bar{R})\xrightarrow{\pi_{\bar{\C P}}}{\rm SL}_2(\mathbb{F}_2)\longrightarrow 1.
\end{equation*}
However, this short exact sequence gives rise to a splitting of ${\rm SL}_2(\bar{R})$ as a semi-direct product $K\rtimes{\rm SL}_2(\mathbb{F}_2):$ 
Namely, as $\C P$ is a ramified prime divisor of $2R$, it follows that $2$ is an element of $\C P^2.$ So in particular, we obtain that $2=0$ holds in $\bar{R}.$ But this then implies that $\mathbb{F}_2$ is a subring of $\bar{R}$ and so ${\rm SL}_2(\mathbb{F}_2)$ is a subgroup of ${\rm SL}_2(\bar{R}).$ Then the inclusion 
$j:{\rm SL}_2(\mathbb{F}_2)\to{\rm SL}_2(\bar{R})$ has the property $\pi_{\bar{\C P}}\circ j={\rm id}_{{\rm SL}_2(\mathbb{F}_2)}$ and this gives us the desired splitting. Next, we state the required proposition:

\begin{proposition}\label{ramified_abelianization}
Let $R=\C O_S$ be the ring of S-algebraic integers of a number field $K$ and $\C P$ a ramified prime divisor of $2R$ with $R/\C P=\mathbb{F}_2.$ Then for the splitting ${\rm SL}_2(\bar{R})=K\rtimes{\rm SL}_2(\mathbb{F}_2)$ described above, the map
\begin{equation*}
K\rtimes{\rm SL}_2(\mathbb{F}_2)\ni\left[
\begin{pmatrix}
a & b\\
c & d
\end{pmatrix},
A\right]\mapsto a-1+b+c
\end{equation*}
defines a group homomorphism $q:{\rm SL}_2(\bar{R})\to\bar{\C P}.$
\end{proposition}

\begin{proof}
We split the proof into two parts: We first show that the three maps 
\begin{equation*}
q_1:K\to\bar{\C P},
\begin{pmatrix}
a & b\\
c & d
\end{pmatrix}
\mapsto b,
q_2:K\to\bar{\C P},
\begin{pmatrix}
a & b\\
c & d
\end{pmatrix}
\mapsto c\text{ and }
q_3:K\to\bar{\C P},
\begin{pmatrix}
a & b\\
c & d
\end{pmatrix}
\mapsto a-1
\end{equation*}
all individually define group homomorphisms. This yields that 
\begin{equation*}
q':K\to\bar{\C P},
\begin{pmatrix}
a & b\\
c & d
\end{pmatrix}
\mapsto a-1+b+c
\end{equation*}
is a group homomorphism. Then secondly, we will show that the map $q'$ is invariant under conjugation by ${\rm SL}_2(\mathbb{F}_2)$ and this will yield that $q'$ extends to the required group homomorphism $q:{\rm SL}_2(\bar{R})\to\bar{\C P}.$ For the first step, let 
\begin{equation*}
A_1=\begin{pmatrix}
a_1 & b_1\\
c_1 & d_1
\end{pmatrix},
A_2=\begin{pmatrix}
a_2 & b_2\\
c_2 & d_2
\end{pmatrix}
\end{equation*}
be elements of $K$ and note that 
\begin{equation*}
A_1\cdot A_2=
\begin{pmatrix}
a_1\cdot a_2+b_1\cdot c_2 & a_1\cdot b_2+b_1\cdot d_2\\
c_1\cdot a_2+d_1\cdot c_2 & c_1\cdot b_2+d_1\cdot d_2
\end{pmatrix}.
\end{equation*}
But note that this implies $q_1(A_1\cdot A_2)=a_1\cdot b_2+b_1\cdot d_2.$ However, $a_1-1$ is an element of $\bar{\C P}$ and so is $b_2.$ Thus 
$(a_1-1)\cdot b_2$ is an element of $\bar{\C P}^2=0$ in $\bar{R}.$ But this yields 
\begin{equation*}
a_1\cdot b_2=(a_1-1+1)\cdot b_2=(a_1-1)\cdot b_2+b_2=b_2.
\end{equation*}
Similarly, we obtain $b_1\cdot d_2=b_1$ and hence 
\begin{equation*}
q_1(A_1\cdot A_2)=a_1\cdot b_2+b_1\cdot d_2=b_2+b_1=q_1(A_2)+q_1(A_1).
\end{equation*}
holds. This shows that $q_1$ is indeed a homomorphism. The claim for $q_2$ follows in essentially the same manner. To show the claim for $q_3$ note that
$q_3(A_1\cdot A_2)=a_1\cdot a_2+b_1\cdot c_2-1.$ But again $b_1$ and $c_2$ are elements of $\bar{\C P}$ and hence $b_1\cdot c_2=0.$ Thus $q_3(A_1\cdot A_2)=a_1\cdot a_2-1.$ Next,  
\begin{align*}
q_3(A_1\cdot A_2)&=a_1\cdot a_2-1=a_1\cdot (a_2-1)+a_1-1=(a_1-1)\cdot(a_2-1)+(a_2-1)+(a_1-1)\\
&=(a_1-1)\cdot(a_2-1)+q_3(A_2)+q_3(A_1).
\end{align*}
However, as $a_1-1$ and $a_2-1$ are both elements of $\bar{\C P},$ we obtain $(a_1-1)\cdot(a_2-1)=0$ and so the claim for $q_3$ follows. Thus the map $q':K\to\bar{\C P}$ indeed defines a homomorphism. For the second claim, that is, that $q'$ is invariant under conjugation by ${\rm SL}_2(\mathbb{F}_2),$ we leave it as an exercise to the reader to check that ${\rm SL}_2(\mathbb{F}_2)$ is generated by the two matrices
\begin{equation*}
X:=\begin{pmatrix}
0 & 1\\
1 & 0
\end{pmatrix}\text{ and }
Y:=
\begin{pmatrix}
1 & 1\\
1 & 0
\end{pmatrix}.
\end{equation*}
Thus it suffices to check that $q'$ is invariant under conjugation by $X$ and $Y.$ To this end, note for $A_1\in K$ as above, that
\begin{equation*}
XA_1X^{-1}
=
\begin{pmatrix}
d_1 & c_1\\
b_1 & a_1
\end{pmatrix}.
\end{equation*}
This implies $q'(XA_1X^{-1})=c_1+b_1+d_1-1.$ Hence to show that $q'(XA_1X^{-1})=q'(A_1),$ it suffices to show that $d_1-1=a_1-1.$ To this end, note that both $a_1-1$ and $d_1-1$ are elements of $\bar{\C P}$ but the ideal $\bar{\C P}$ only has a single non-trivial element $\bar{t}$. But assume wlog. that $a_1-1$ is $0$ and $d_1-1$ is $\bar{t}.$ Then note that as $A_1$ is an element of ${\rm SL}_2(\bar{R})$, we obtain
\begin{equation*}
1=a_1\cdot d_1-b_1\cdot c_1=1\cdot(1+\bar{t})-b_1\cdot c_1=1+\bar{t}-b_1\cdot c_1.
\end{equation*}
However, $b_1\cdot c_1=0$ holds and thus $\bar{t}=0$ follows. But this is impossible. Thus $a_1-1$ and $d_1-1$ must necessarily be the same element of $\bar{\C P}$ and hence $q'$ is invariant under conjugation by $X.$ For the invariance under $Y$, note that 
\begin{equation*}
YA_1Y^{-1}=
\begin{pmatrix}
1 & 1\\
1 & 0
\end{pmatrix}
\cdot 
\begin{pmatrix}
a_1 & b_1\\
c_1 & d_1
\end{pmatrix}
\cdot Y^{-1}
=
\begin{pmatrix}
a_1+c_1 & b_1+d_1\\
a_1 & b_1
\end{pmatrix}
\cdot 
\begin{pmatrix}
0 & 1\\
1 & 1
\end{pmatrix}
=
\begin{pmatrix}
b_1+d_1 & a_1+b_1+c_1+d_1\\
b_1 & a_1+b_1
\end{pmatrix}.
\end{equation*} 
But this implies
\begin{align*} 
q'(YA_1Y^{-1})&=(a_1+b_1+c_1+d_1)+b_1+(b_1+d_1-1)=a_1+3b_1+c_1+2d_1-1=b_1+c_1+a_1-1\\
&=q'(A_1).
\end{align*}
Here, we used the aforementioned fact that $2=0$ holds in $\bar{R}.$ This finishes the proof.
\end{proof}

\begin{remark}
One can actually show that $\bar{R}=\mathbb{F}_2[T]/(T^2)$ and that ${\rm SL}_2(\bar{R})=\mathbb{Z}_2^3\rtimes S_3$ with the permutation group $S_3$ operating on 
$\mathbb{Z}_2^3$ by permutation of components. Then $q:{\rm SL}_2(\bar{R})\to\mathbb{Z}_2$ has the form $q(x,\sigma)=x_1+x_2+x_3$ for $x\in\mathbb{Z}_2^3$ and 
$\sigma\in S_3.$
\end{remark}

Note, that this implies:

\begin{corollary}\label{ramified_cor}
Let $R=\C O_S$ be the ring of S-algebraic integers of a number field $K$ and $\C P$ a ramified prime divisor of $2R$ with $R/\C P=\mathbb{F}_2$ and let $L\geq 2$ be given. Then there is a unique epimorphism ${\rm SL}_2(R/\C P^L)\to R/\C P^2$ mapping $E_{12}(x+\C P^L)$ to $x+\C P^2$ for all $x\in R.$ This unique epimorphism is the abelianization map of ${\rm SL}_2(R/\C P^L).$
\end{corollary}

\begin{proof}
Due to Lemma~\ref{factoring_ramified}, we may assume $L=2.$ Next, note that according to the afore-mentioned splitting ${\rm SL}_2(R/\C P^2)=K\rtimes{\rm SL}_2(\mathbb{F}_2),$ there is an epimorphism
${\rm SL}_2(R/\C P^2)\to{\rm SL}_2(\mathbb{F}_2).$ Then one can easily check that 
\begin{align*}
{\rm SL}_2(\mathbb{F}_2)\ni E_{12}(1)\mapsto 1\in\mathbb{F}_2\text{ and }E_{21}(1)\mapsto 1
\end{align*}
extends to an epimorphism ${\rm SL}_2(\mathbb{F}_2)\to\mathbb{F}_2$ and so there is an epimorphism $h:{\rm SL}_2(R/\C P^2)\to\mathbb{F}_2$ with $h(E_{12}(1))=1.$
Second, consider the homomorphism $q:{\rm SL}_2(R/\C P^2)\to\C P/\C P^2=\mathbb{F}_2$ from Proposition~\ref{ramified_abelianization} and consider the homomorphism
$h\oplus q:{\rm SL}_2(R/\C P^2)\to\mathbb{F}_2\oplus\mathbb{F}_2.$ 
Note that $h\oplus q(E_{12}(1))=(1,0)$. For $\bar{t}$ the unique non-trivial element of $\C P/\C P^2,$ we have further $h\oplus q(E_{12}(\bar{t}))=(0,1).$ Thus 
$h\oplus q$ is an epimorphism onto $\mathbb{F}_2\oplus\mathbb{F}_2$. But we know from Lemma~\ref{factoring_ramified} that the abelianization of ${\rm SL}_2(R/\C P^L)$ has at most four elements and is a quotient of $R/\C P^2.$ This finishes the proof using the isomorphism $\mathbb{F}_2\oplus\mathbb{F}_2=R/\C P^2$ sending $(1,0)$ to $1+\C P^2$ and $(0,1)$ to $\bar{t}.$ 
\end{proof}

Last, we will determine the abelianization of ${\rm SL}_2(R/\C K_l^{k_l}).$ To this end, note

\begin{lemma}\label{factoring_3divisors}
Let $R=\C O_S$ be the ring of S-algebraic integers of a number field $K$, let $\C K$ be a prime divisor of $3R$ with $R/\C K=\mathbb{F}_3$ and let $L$ be a natural number. Then the abelianization homomorphism of ${\rm SL}_2(R/\C K^L)$ factors through ${\rm SL}_2(R/\C K)={\rm SL}_2(\mathbb{F}_3)$. 
\end{lemma}

\begin{proof}
Choose an element $x\in\C K$ and let $\bar{x}$ be its image in $\bar{R}:=R/\C K^L.$ Then $1+\bar{x}$ is a unit in $\bar{R}.$ Further, let $\bar{a}\in\bar{R}$ be arbitrary. Then
\begin{equation*}
[h(1+\bar{x}),E_{12}(\bar{a})]=E_{12}((1+\bar{x})^2\bar{a})\cdot E_{12}(-\bar{a})=E_{12}((2\bar{x}+\bar{x}^2)\cdot\bar{a})=E_{12}(\bar{x}(2+\bar{x})\cdot\bar{a})
\end{equation*}
is an element of the commutator subgroup $[{\rm SL}_2(\bar{R}),{\rm SL}_2(\bar{R})].$ But note that $2\notin\C K$ and so $2+\bar{x}$ is a unit of $\bar{R}.$ So in particular, the commutator subgroup of ${\rm SL}_2(\bar{R})$ contains the normal subgroup $N:=\dl E_{12}(\bar{a})\mid\bar{a}\in\bar{\C K}\dr$ for $\bar{\C K}$ the image of $\C K$ in $\bar{R}.$ But according to \cite[p.240]{MR0249491}, this normal subgroup is the kernel of the epimorphism $\pi_{\bar{\C P^2}}:{\rm SL}_2(\bar{R})\to{\rm SL}_2(\bar{R}/\bar{\C K})={\rm SL}_2(R/\C K)={\rm SL}_2(\mathbb{F}_3).$ Thus the abelianization homomorphism of ${\rm SL}_2(\bar{R})$ indeed factors through ${\rm SL}_2(R/\C K)$.
\end{proof}

But we have the following lemma:

\begin{lemma}\label{ab_3}
There is an epimorphism ${\rm SL}_2(\mathbb{F}_3)\to\mathbb{F}_3$ mapping $E_{12}(1)$ to $1$.
\end{lemma}

We skip the proof, because it is similar to the proof of \cite[Lemma~6.5]{General_strong_bound}. But similar to the proof of Proposition~\ref{abelianization_umramified} or Corollary~\ref{ramified_cor}, one sees that the abelianization of ${\rm SL}_2(\mathbb{F}_3)$ is a cyclic group with at most three elements and so the abelianization homomorphism of ${\rm SL}_2(R/\C K_l^{k_l})$ is the unique epimorphism ${\rm SL}_2(R/\C K_l^{k_l})\to R/\C K_l$ mapping $E_{12}(x+\C K_l^{k_l})$ to $x+\C K_l$ for all $x\in R.$ To summarize the discussion:

\begin{corollary}\label{abelianization}
Let $R=\C O_S$ be the ring of S-algebraic integers of a number field $K$ such that $R$ has infinitely many units and  
\begin{enumerate}
\item{let $\C P_1,\dots,\C P_{r_2(R)}$ be the ramified prime divisors of $2R$ with $R/\C P_j=\mathbb{F}_2$,}
\item{let $\C Q_1,\dots,\C Q_{r_1(R)}$ be the unramified prime divisors of $2R$ with $R/\C Q_i=\mathbb{F}_2$ and} 
\item{let $\C K_1,\dots,\C K_{q(R)}$ be the prime divisors of $3R$ with $R/\C K_l=\mathbb{F}_3$}.
\end{enumerate}
Then the abelianization homomorphism of ${\rm SL}_2(R)$ is the unique epimorphism
\begin{equation*}
{\rm SL}_2(R)\to(\prod_{j=1}^{r_2(R)}R/\C P_j^2)\times(\prod_{i=1}^{r_1(R)}R/\C Q_i^2)\times(\prod_{l=1}^{q(R)}R/\C K_l)=\mathbb{F}_2^{2r_2(R)}\times\mathbb{Z}_4^{r_1(R)}\times\mathbb{F}_3^{q(R)} 
\end{equation*}
which maps $E_{12}(x)$ to $(\prod_{j=1}^{r_2(R)}(x+\C P_j^2),\prod_{i=1}^{r_1(R)}(x+\C Q_i^2),\prod_{l=1}^{q(R)}(x+\C K_l))$ for all $x\in R.$
\end{corollary}

\begin{proof}
This follows immediately from the determination of the abelianization maps of ${\rm SL}_2(R/\C Q_i^{q_i})$, ${\rm SL}_2(R/\C P_j^{p_j})$ and ${\rm SL}_2(R/\C K_l^{k_L})$ done above. 
\end{proof}

\begin{remark}
Note in particular, that the abelianization of ${\rm SL}_2(R/\C P_i^2)$ and ${\rm SL}_2(R/\C Q_j^2)$ both have $4$ elements, whereas the abelianization of $S_3={\rm SL}_2(\mathbb{F}_2)={\rm SL}_2(R/\C P_i)={\rm SL}_2(R/\C Q_j)$ only has $2$ elements. Thus the multiplicities $p_i$ and $q_j$ in the ideal $I$ mentioned before can not be chosen smaller as $2.$  
\end{remark}

But we can prove the second part of Theorem~\ref{lower_bounds_sl2} now:

\begin{proof}
Assume for contradiction that there is a finite, normally generating subset $T$ of ${\rm SL}_2(R)$ with $|T|\leq v(R)-1.$ But $T$ must map onto a generating set of the abelianization $\mathbb{F}_2^{2r_2(R)}\times\mathbb{Z}_4^{r_1(R)}\times\mathbb{F}_3^{q(R)}$ of ${\rm SL}_2(R).$ However, the rank of the abelian group $\mathbb{F}_2^{2r_2(R)}\times\mathbb{Z}_4^{r_1(R)}\times\mathbb{F}_3^{q(R)}$ is $\max\{2r_2(R)+r_1(R),q(R)\}=v(R)$ as $2$ and $3$ are distinct primes. This is a contradiction.  
\end{proof}

\subsection{The proof of the first part of Theorem~\ref{lower_bounds_sl2}}

For the sake of a more uniform proof, we assume in the following that for $r_1:=r_1(R),r_2:=r_2(R)$ and $q:=q(R)$, one has $2r_2+r_1=q=v(R)$ and $r_1,q\geq 1$. Next, due to the Chinese Remainder Theorem there is an epimorphism
\begin{align*}
&R\to\left(\prod_{j=1}^{r_2} R/\C P_j^2\right)\times\left(\prod_{i=1}^{r_1}R/\C Q_i^2\right)\times\left(\prod_{l=1}^q R/\C K_l\right),\\
&x\mapsto
\left[\left(\prod_{j=1}^{r_2}x+\C P_j^2\right),\left(\prod_{i=1}^{r_1}x+\C Q_i^2\right),\left(\prod_{l=1}^q x+\C K_l\right)\right]
\end{align*}
Next, choose elements $c_1,\dots,c_{r_2}\in R$ which under the epimorphism given by the Chinese Remainder Theorem map as follows:  
For each $1\leq u\leq r_2$, we can choose $c_u'\in R$ which maps to an element of $\C P_u/\C P_u^2,$ which is not $0.$ Then we choose the $c_u\in R$ as mapping to
\begin{equation*}
(c'_u+\C P_u^2,0+\prod_{j=1,j\neq u}^{r_2}\C P_j^2,0+\prod_{i=1}^{r_1}\C Q_i^2,1+\C K_{u+r_1+r_2},0+\prod_{l=1,l\neq u+r_1+r_2}^q\C K_l).
\end{equation*} 

Also for $k\geq\max\{2r_2+r_1,q\}=2r_2+r_1=q=v(R),$ choose elements $y_1,\dots,y_k$ of $R$ with distinct prime divisors $\C Y_1,\dots,\C Y_k$, that is for $1\leq f\leq k$ the only prime divisor of the ideal $y_fR$ is $\C Y_f$ and such that none of the $\C Y_f$ agree with any of the prime divisors of the elements $c_1,\dots,c_{r_2},$ each other or the prime ideals $\C P_1,\dots,\C P_{r_2},\C Q_1,\dots,\C Q_{r_1},\C K_1,\dots,\C K_q.$ Next, let $C$ be the class number of $R$. 
\begin{enumerate}
\item{For each $1\leq u\leq r_1$, choose a generator $s_u$ of the principal ideal 
$(\C Y_1\cdots\hat{\C Y_u}\cdots\C Y_k\cdot\C P_1\cdots\C P_{r_2}\cdot\C Q_1\cdots\hat{\C Q_u}\cdots\C Q_{r_1}\cdot\C K_1\cdots\hat{\C K_u}\cdots\C K_q)^C$ with the hats denoting the omission of the corresponding prime factors.} 
\item{For $r_1+1\leq u\leq r_1+r_2$, choose a generator $s_u$ of the principal ideal $(\C Y_1\cdots\hat{\C Y_u}\cdots\C Y_k\cdot\C P_1\cdots\hat{\C P_u}\cdots\C P_{r_2}\cdot\C Q_1\cdots\C Q_{r_1}\cdot\C K_1\cdots\hat{\C K_u}\cdots\C K_q)^C.$}
\item{For $r_2+r_1+1\leq u\leq r_1+2r_2$, consider a generator $t_u$ of the principal ideal $(\C Y_1\cdots\hat{\C Y_u}\cdots\C Y_k\cdot\C P_1\cdots\C P_{r_2}\cdot\C Q_1\cdots\C Q_{r_1}\cdot\C K_1\cdots\hat{\C K_u}\cdots\C K_q)^C$ and set $s_u:=c_{u-r_1-r_2}\cdot t_u$.}
\item{For $2r_2+r_1+1\leq u\leq k$, set $s_u:=\prod_{z=1,z\neq u}^k y_z$.}
\end{enumerate}

Finally, consider the set $T:=\{E_{12}(s_u)\mid 1\leq u\leq k\}.$ Next, we are going to check that $T$ satisfies the two necessary conditions for a normally generating subset of ${\rm SL}_2(R)$ stated in Corollary~\ref{normal_gen_criterion}. First, to prove $\Pi(T)=\emptyset,$ note 
\begin{align*}
&\Pi(\{E_{12}(s_u)\})=\{\C Y_f\mid f\neq u\}\cup\{\C P_f\}\cup\{\C Q_f\mid f\neq u\}\cup\{\C K_f\mid f\neq u\}\text{ for }1\leq u\leq r_1,\\
&\Pi(\{E_{12}(s_u)\})=\{\C Y_f\mid f\neq u\}\cup\{\C P_f\mid f\neq u-r_1\}\cup\{\C Q_f\}\cup\{\C K_f\mid f\neq u\}\text{ for }r_1+1\leq u\leq r_1+r_2,\\
&\Pi(\{E_{12}(s_u)\})=\{\C Y_f\mid f\neq u\}\cup\{\C K_f\mid f\neq u\}\cup\{\text{prime divisors of }c_{u-r_1-r_2}R\}\\
&\text{ for }r_1+r_2+1\leq u\leq r_1+2r_2\text{ and}\\
&\Pi(\{E_{12}(s_u)\})=\{\C Y_f\mid f\neq u\}\text{ for }r_1+2r_2+1\leq u\leq k. 
\end{align*}
But we have chosen all the $\C Y_f$ to be distinct from each other and from the prime divisors of the elements $c_1,\dots,c_{r_2}.$ Thus the only prime divisors of $c_uR$ for $1\leq u\leq r_2$ that can possibly be contained in $\Pi(T)=\bigcap_{i=1}^k\Pi(\{E_{12}(s_i)\})$ are elements of 
\begin{equation*}
\{\C P_1,\dots,\C P_{r_2}\}\cup\{\C Q_1,\dots,\C Q_{r_1}\}\cup\{\C K_f\mid f\neq u+r_1+r_2\}.
\end{equation*}
Thus we obtain $\Pi(T)=\bigcap_{i=1}^k\Pi(\{E_{12}(s_i)\})=\emptyset.$

Next, we have to show that $T$ maps onto a generating set of the abelianization of ${\rm SL}_2(R).$ Note that the abelianization map of ${\rm SL}_2(R)$ factors through the quotient
\begin{equation*}
(\prod_{j=1}^{r_2}{\rm SL}_2(R/\C P_j^2))\times(\prod_{i=1}^{r_1}{\rm SL}_2(R/\C Q_i^2))\times(\prod_{l=1}^q{\rm SL}_2(R/\C K_l))=:G.
\end{equation*}
First, observe that for each $u\leq 2r_2+r_1=q,$ the only non-trivial component of the image of $E_{12}(s_u)^4=E_{12}(4\cdot s_u)$ in $G$ is the ${\rm SL}_2(R/\C K_u)$-component. However, both $4$ and $s_u$ map to units in $R/\C K_u=\mathbb{F}_3$. Hence $E_{12}(s_u)^4$ maps onto a generator of the $R/\C K_u=\mathbb{F}_3$-component of the abelianization of ${\rm SL}_2(R)$ according to Corollary~\ref{abelianization}. Hence we may assume for simplicity that there are no prime ideals of the type $\C K_u$ anymore.

Second, consider an element $E_{12}(s_u)$ for $1\leq u\leq r_1.$ Now, the only non-trivial component of $E_{12}(s_u)$ after mapping to $G$ is the ${\rm SL}_2(R/\C Q_u^2)$-component. This component is $E_{12}(s_u+\C Q_u^2)$ and due to the choice of $s_u$ the element $s_u+\C Q_u^2$ is a unit in $R/\C Q_u^2=\mathbb{Z}_4.$ However, we know from Corollary~\ref{abelianization} that $E_{12}(s_u)$ gets mapped to the element $s_u+\C Q_u^2$ of the $R/\C Q_u^2$-component of the abelianization of ${\rm SL}_2(R)$ and clearly units in $\mathbb{Z}_4$ generate $(\mathbb{Z}_4,+).$  

Last, observe that for $r_1+1\leq u\leq r_1+r_2$ the only non-trivial components of $E_{12}(s_u)$ and $E_{12}(s_{u+r_2})$ after mapping to $G$ are the 
${\rm SL}_2(R/\C P_{u-r_1}^2)$-component. Due to the choice of $s_u$ and $s_{u+r_2},$ we know that $s_u+\C P_{u-r_1}^2$ is a unit of $R/\C P_{u-r_1}^2$ and 
$s_{u+r_2}+\C P_{u-r_1}^2$ is a non-trivial element of $\C P_{u-r_1}/\C P_{u-r_1}^2$. But these two elements generate $(R/\C P_{u-r_1}^2,+)$ and according to Corollary~\ref{abelianization} the elements $E_{12}(s_u)$ and $E_{12}(s_{u+r_2})$ map to the elements $s_u+\C P_{u-r_1}^2$ and $s_{u+r_2}+\C P_{u-r_1}^2$ in the $R/\C P_{u-r_1}^2$-component of the abelianization of ${\rm SL}_2(R)$ respectively. So to summarize, $T$ maps onto a generating set of the abelianization of ${\rm SL}_2(R).$

So we have proven that $T$ does in fact normally generate ${\rm SL}_2(R).$ But $T$ has $k$ elements and so $\Delta_k({\rm SL}_2(R))\geq\|{\rm SL}_2(R)\|_T$ holds. But similar to the proofs of \cite[Theorem~6.3]{General_strong_bound} or \cite[Corollary~6.2]{KLM}, it is now easy to check that indeed $\|{\rm SL}_2(R)\|_T\geq 2k$ holds. Thus we have shown $\Delta_k({\rm SL}_2(R))\geq 2k$ for all $k\geq\max\{2r_2+r_1,q\}=v(R)$ at least assuming that $r_1,q\geq 1$ and $r_1+2r_2=q$. The other cases work in essentially the same way though with merely slight differences concerning the precise definition of the various $s_u.$ This finishes the proof of the first part of Theorem~\ref{lower_bounds_sl2}.

\begin{remark}
In \cite[Theorem~3]{explicit_strong_bound_sp4_pseudo_good}, we showed that the presence of bad primes of a ring of S-algebraic integers $R$ resulted in the better lower bound $4k+r_1(R)+r_2(R)$ on $\Delta_k({\rm Sp}_4(R))$, essentially because the conjugacy diameter of ${\rm Sp}_4(\mathbb{F}_2)=S_6$ is $5$, rather than $4.$ One can use a similar argument to improve the lower bound on $\Delta_k({\rm SL}_2(R))$ for $k\geq v(R)$ from $2k$ to $2k+r_2(R).$ This is due to $\Delta_2({\rm SL}_2(\mathbb{F}_2[T]/(T^2)))=\Delta_2(\mathbb{F}_2^3\rtimes S_3)\geq 3.$ Curiously, neither the presence of unramified prime divisors $\C Q$ of $2R$ with $R/\C Q=\mathbb{F}_2$ nor of prime ideals $\C K$ with $R/\C K=\mathbb{F}_3$ results in better lower bounds on $\Delta_k({\rm SL}_2(R))$, because both $G_1:={\rm SL}_2(R/\C Q)={\rm SL}_2(\mathbb{F}_2)=S_3$ and $G_2:={\rm SL}_2(R/\C K)={\rm SL}_2(\mathbb{F}_3)$ satisfy $\Delta_1(G_i)=2$ for $i=1,2.$
\end{remark}

We also want to note the following: 

\begin{corollary}\label{square_root_v_values}
Let $D$ be a positive, square-free integer and $R$ the ring of algebraic integers in $\mathbb{Q}[\sqrt{D}].$ Then the value of $v(R)=\max\{2r_2(R)+r_1(R),q(R)\}$ is 
\begin{enumerate}
\item{$v(R)=0$ precisely if $D\equiv 5\text{ mod }8$ and $D\equiv 2\text{ mod }3$, so $\Delta_1({\rm SL}_2(R))>-\infty.$}
\item{$v(R)=1$ precisely if $D\equiv 5\text{ mod }8$ and $D\equiv 0\text{ mod }3$, so $\Delta_1({\rm SL}_2(R))>-\infty.$}
\item{$v(R)=2$ precisely if $D\equiv 1,2,3,6,7\text{ mod }8$ or $D\equiv 1\text{ mod }3,$ so $\Delta_1({\rm SL}_2(R))=-\infty.$}
\end{enumerate}
\end{corollary}

\begin{proof}
We obtain from \cite[Theorem~25]{MR3822326} that the ideal $2R$ splits and ramifies in $R$ as follows:
\begin{enumerate}
\item{$2R$ is inert precisely if $D\equiv 5\text{ mod }8.$}
\item{$2R$ ramifies precisely if $D\equiv 2,3,6,7\text{ mod }8.$}
\item{$2R$ splits precisely if $D\equiv 1\text{ mod }8.$}
\end{enumerate}
We first deal with the case of $2R$ not inert. If $2R$ ramifies, then $r_2(R)=1, r_1(R)=0$ and $q(R)$ can not be bigger than $2$ and so $v(R)=2$ holds. Similarly, the splitting of $2R$ implies that $v(R)=r_1(R)=2.$
However, if $2R$ is inert, then the field of residue $R/2R$ is $\mathbb{F}_4$ and so $r_1(R)=r_2(R)=0$ and hence $v(R)=q(R).$ Then we have to distinguish two cases: 
Either $3$ is a divisor of $D$ or not. If $3$ divides $D,$ then \cite[Theorem~25]{MR3822326} implies further that $3R$ is a ramified prime in $R$ and so $q(R)=1$ and $v(R)=1$ hold in this case. If $3$ is not a prime divisor of $D,$ then \cite[Theorem~25]{MR3822326} implies two further sub-cases: $3R$ splits in $R$ if and only if 
$D$ is a quadratic residue in $\mathbb{F}_3$ and otherwise $3R$ is inert in $R.$ However, considering that $3$ does not divide $D,$ one can check that $D$ is a quadratic residue precisely if $D\equiv 1\text{ mod }3.$ So $v(R)=q(R)=2$ if $D\equiv 1\text{ mod }3$ and $v(R)=q(R)=0$ if $D\equiv 2\text{ mod }3.$
Finally, the claim of the corollary follows from Theorem~\ref{lower_bounds_sl2} together with the determined values of $v(R).$
\end{proof}

\section{Closing remarks}

This paper finishes the qualitative discussion of strong boundedness of S-arithmetic, split Chevalley groups. However, the question remains how precisely $\Delta_k(G(\Phi,R))$ for $R$ the ring of S-algebraic integers in a number field $K$ and $\Phi$ an irreducible root system depends on $R$ and $\Phi.$ The previous papers \cite{KLM} and \cite{explicit_strong_bound_sp_2n} showed that $\Delta_k(G(\Phi,R))$ has a lower bound proportional to $k\cdot{\rm rank}(\Phi)$ and an upper bound proportional to $k\cdot{\rm rank}(\Phi)^2$ with proportionality factors independent of $R$ at least in the case of $R$ having infinitely many units. However, as mentioned in \cite[Remark~5.14(2)]{General_strong_bound} it seems likely to us that $\Delta_k(G(\Phi,R))$ also has an upper bound proportional to $k\cdot{\rm rank}(\Phi)$ and in future research we will study this further.

%Another open problem remaining is the question of how large a class of non-uniform lattices admit the property of strong boundedness. Current proofs rely strongly on bounded generation and it would be interesting to see whether strong boundedness can be proven independently of it.

\bibliography{bibliography}
\bibliographystyle{plain}

\end{document}